\documentclass[12pt,a4paper,psamsfonts]{amsart}
\usepackage[english]{babel}
\usepackage[T1]{fontenc}
\usepackage{mathtools}
\usepackage{fancyhdr}
\usepackage{appendix}
\usepackage{amssymb,amscd,amsxtra,calc}
\usepackage{mathrsfs}
\usepackage{cmmib57}
\usepackage{multirow}
\usepackage[all]{xy}
\usepackage{longtable}
\usepackage[colorlinks=true,linkcolor=blue,anchorcolor=blue,citecolor=blue,pagebackref,linktocpage]{hyperref}
\usepackage{cleveref}
\usepackage{tikz}
\usepackage{cite}
\theoremstyle{plain}
    \newtheorem{thm}{Theorem}[section]

     \newtheorem{conjecture}[thm]{Conjecture}

    \newtheorem{lemma}[thm]{Lemma}
    \newtheorem{proposition}[thm]{Proposition}
    
    \newtheorem{theorem}[thm]{Theorem}
    \newtheorem{problem}[thm]{Problem}

\theoremstyle{definition}
    \newtheorem{definition}[thm]{Definition}

    \newtheorem*{notation*}{Notation and Terminology}
      
    \newtheorem{remark}[thm]{Remark}
    
\theoremstyle{remark}

\newcommand{\arxiv}[1]{\href{https://arxiv.org/abs/#1}{{\tt arXiv:#1}}}

\newcommand{\bP}{\mathbb{P}}
\newcommand{\bQ}{\mathbb{Q}}
\newcommand{\bR}{\mathbb{R}}
\newcommand{\bC}{\mathbb{C}}

\newcommand{\bZ}{\mathbb{Z}}

\newcommand{\NE}{\overline{\operatorname{NE}}}
\newcommand{\Nef}{\operatorname{Nef}}

\newcommand{\NS}{\operatorname{NS}}

\newcommand{\Pic}{\operatorname{Pic}}

\newcommand{\mstriangle}[1]{
\begin{tikzpicture}[x=0.3cm,y=0.3cm]
\draw (-0.4,-0.433) -- (1.4,-0.433);
\draw (-0.2,-0.7794) -- (0.7,0.7794);
\draw (1.2,-0.7794) -- (0.3,0.7794);
\end{tikzpicture}
}
\newcommand{\mssharp}[1]{
\begin{tikzpicture}[x=0.3cm,y=0.3cm]
\draw (-0.8,-0.5) -- (0.8,-0.5);
\draw (-0.8,0.5) -- (0.8,0.5);
\draw (-0.5,-0.8) -- (-0.5,0.8);
\draw (0.5,-0.8) -- (0.5,0.8);
\end{tikzpicture}
}

\makeatletter

\newcommand{\Rmnum}[1]{\expandafter\@slowromancap\romannumeral #1@}
\makeatother

\title[Fake quadrics with bounded cohomology property]
{Fake quadrics with bounded cohomology property}

\author{Sichen Li}
\address{
School of Mathematics, East China University of Science and Technology, Shanghai 200237, P. R. China}
\email{\href{mailto:sichenli@ecust.edu.cn}{sichenli@ecust.edu.cn}}


\begin{document}

\begin{abstract}
In this paper, we characterize fake quadrics $X$ satisfying the bounded cohomology property.
Namely, there exists a constant $c_X>0$ such that $h^1(\mathcal O_X(C))\le c_X h^0(\mathcal O_X(C))$ for every curve $C$ on  $X$.
\end{abstract}

\maketitle
\section{Introduction}
A very open problem in the theory of algebraic surfaces is the following bounded negativity conjecture (BNC for short).
\begin{conjecture}
\cite[Conjecture 1.1]{Bauer et al. 2013}
Let $X$ be a smooth projective surface over $\bC$.
Then there exists an integer $b=b(X)\ge0$ such that $C^2\ge-b$ for every curve $C\subseteq X$.
\end{conjecture}
Bauer et al. \cite[Conjecture 2.5.3]{Bauer et al. 2012} introduced the following terminology.
\begin{definition}
\label{BCP-defn}
A smooth projective surface $X$ is said to satisfy the bounded cohomology property (BCP for short), if there exists a constant $c_X>0$ such that $h^1(\mathcal O_X(C))\le c_Xh^0 (\mathcal O_X(C))$ for every curve $C$ on $X$.
\end{definition}
It turns out by Ciliberto et al. \cite{Ciliberto et al. 2017} that the BCP implies the BNC.
Note that a smooth projective surface satisfies the BCP provided that  $-K_X$ is pseudoeffective (cf.\cite[Proposition 1.4]{HL26}).
In general, it is  an open problem to classify smooth projective surfaces with the BCP as posed by Ciliberto et al.\cite[Question 6]{Ciliberto et al. 2017}.
\begin{problem}
Classify all smooth projective surfaces with the BCP.
\end{problem}
\begin{remark}
For recent progress of the BCP, we refer to \cite{Ciliberto et al. 2017,  HL26, LL26, Li21, Li23, Li26a, Li26b}.
For some examples of surfaces with the BCP, we refer to \cite[Section 5]{HL26}.
\end{remark}
In \cite{HK00}, Hu and Keel proved that if the Cox ring of an algebraic variety $X$ is finitely generated, then $X$ enjoys particularly nice properties with respect to the minimal model program.
Such varieties are called Mori dream spaces.
Note that  a smooth projective surface $X$ with $q(X)=0$ is a Mori dream surface if and only if its effective cone is rational polyhedral and every nef divisor is semiample; see \cite[Corollary~2.6]{AHL10}, also see  \cite[Chapter 5]{ADHL15}.
The classification of Mori dream surfaces is a widely open problem.
Over the past decade or so, substantial progress has been made toward the classification of Mori dream surfaces \cite{AHL10, AL11, KL19,KL26, LLL26, LL26, TVV11}.

As a first step toward the classification of Mori dream surfaces $X$ with Picard number $\rho(X)=2$, it is straightforward to derive the following result.
\begin{proposition}
\label{MDS-rho2-prop}
Let $X$ be a smooth projective surface  with $q(X)=0$ such that $\NE(X)=\bR_{\ge0}[C_1]+\bR_{\ge0}[C_2]$, and $C_1$ and $C_2$ are curves on $X$.
Then $X$ is a Mori dream surface if and only if $X$ falls into one of the following cases.
\begin{enumerate}
\item $\kappa(X,C_1)=\kappa(X,C_2)=1$.
\item $\kappa(X,C_1)=1, C_2^2<0$ and $(-C_2^2)C_1+(C_1\cdot C_2)C_2$ is semiample.
\item $C_1^2<0$ and $C_2^2<0$, and $(C_1\cdot C_2)C_1+(-C_1^2)C_2$ and $(-C_2^2)C_1+(C_1\cdot C_2)C_2$ are semiample.
\end{enumerate}
\end{proposition}
\begin{remark}
Determining whether a nef divisor $D$ on a smooth projective surface $X$ is semiample typically requires explicit knowledge of its birational geometry; see, e.g., \cite[Proposition~2.6]{LLL26} and \cite[Theorem~2.9]{LL26} for semiampleness criteria.
To the best of our knowledge, we confirm the existence of each case in Proposition \ref{MDS-rho2-prop}.
	The existence of Case (1) follows from the fact that a surface $X$ is a Mori dream surface provided that it is isogenous to a higher product with $p_g=0$ by  \cite[Proposition 3.4]{KL19} and \cite[Theorem 3.5]{FL22}.
	For case (2), given any $\chi\ge3$, there exists a Jacobian elliptic Mori dream  surface $X$ with $\chi(\mathcal O_X)=\chi, \rho(X)=2$ and a $(-\chi)$-curve by \cite[Proposition 2.3 and Theorem 6.2]{LLL26}.
The classification of smooth projective surfaces $X$ with $\rho(X)=2$ and two negative curves was given in  \cite[Claim 2.11]{Li19}.
Note that a K3 surface is a Mori dream surface if and only if $\NE(X)$ is rational polyhedral by \cite[Theorem 5.1.5.1]{ADHL15}.
Thus, for Case (3), there exist infinitely many Mori dream K3 surfaces  $X$ with $\rho(X)=2$ and two negative curves, see e.g. \cite[Claim 2.12]{Li19}.
\end{remark}
Minimal surfaces $X$ of general type with $p_g(X)=0$ and $K_X^2=8$ are called fake quadrics.
In \cite[Question 3.2]{KL19}, Keum and Lee asked whether a fake quadric is a Mori dream surface.
On the one hand, Cascini et al. \cite{CCCK25} showed that Shavel-type surfaces are even-type fake quadrics, which are not Mori dream surfaces, as their closed Mori cone fails to be rationally polyhedral.
On the other hand, every fake quadric is a Mori dream surface provided that it is isogenous to a higher product by Keum and Lee \cite{KL19} and Frapporti and Lee \cite{FL22}.
Recently, Hua and Li  proved in \cite[Theorem 1.9]{HL26} that every Mori dream surface satisfies the BCP.
These results motivate us to address the question of whether every fake quadric satisfies the BCP even when it is not a Mori dream surface, including cases where its Mori cone fails to be rational polyhedral. 

The BCP for fake quadrics was investigated by Yang~\cite{Yang23}.
In particular, he established the BCP for even fake quadrics and for odd fake quadrics with $b(X)>0$.
For odd fake quadrics with \(b(X)=0\), his treatment relies on the additional assumption that \(X\) is a Mori dream surface.

In the present paper, we give a complete treatment of the BCP for fake quadrics via a different approach, building on methods developed in our previous work~\cite{Li21, Li23, HL26, Li26a, Li26b}. In particular, our approach does not require the Mori dream surface assumption in the odd case.
To formulate our main theorem, we introduce the following terminology.
\begin{definition}
Let \(X\) be a fake quadric such that
$
\NS(X)=\mathbb Z[Q_1]\oplus\mathbb Z[Q_2].
$
We say that a curve \(D\subset X\) is \emph{of \(n\)-type} if
$
D\sim (n+1)Q_1+nQ_2.
$
Set
$$
I_X=\bigl\{n\in\mathbb Z_{>0}\mid
\text{there exists a curve }D\subset X
\text{ of \(n\)-type}\bigr\}.
$$
For each \(n\in I_X\), we fix one \(n\)-type curve representing the class \((n+1)Q_1+nQ_2\).
We say that \(X\) satisfies \emph{uniform \(h^0\)-growth for curves of \(n\)-type} if either \(I_X\) is finite, or \(I_X\) is infinite and
$$
\liminf_{\substack{n\in I_X\\ n\to+\infty}}
\frac{h^0\bigl(\mathcal O_X((n+1)Q_1+nQ_2)\bigr)}{n}>0.
$$
\end{definition}
When $X$ is an odd fake quadric with $b(X)=0$ and $I_X$ is infinite, we obtain a necessary and sufficient condition for the BCP in terms of uniform $h^0$-growth for curves of $n$-type.
We now state our main result in this paper.
\begin{theorem}
\label{Mainthm}
Let \(X\) be a fake quadric.
Then the following statements hold.
\begin{enumerate}
	\item If $X$ is of even type, then $X$ satisfies the BCP.
	\item If $X$ is of odd type with $b(X)>0$, then $X$ satisfies the BCP.
 \item If $X$ is of odd type with $b(X)=0$ and $I_X$ is finite, then $X$ satisfies the BCP.
	\item If $X$ is of odd type with $b(X)=0$ and $I_X$ is infinite, then $X$ satisfies the BCP if and only if $X$ satisfies uniform $h^0$-growth for curves of $n$-type.
\end{enumerate}
\end{theorem}
\begin{remark}
In \cite[Theorem 1.6]{Li26a}, Li characterized the BCP for smooth projective surfaces $X$ with $\kappa(X)=\rho(X)=2$ and rational polyhedral Mori cone.
In contrast, our Theorem \ref{Mainthm} does not assume that the closed Mori cone is rational polyhedral.
Catanese \cite{Catanese25} showed the existence of odd fake quadrics, each of which is isogenous to a product of unmixed type.
A related open question asks whether there exists an odd fake quadric $X$ with $b(X)>0$; see \cite[Section 8]{CCCK25} for further open problems concerning fake quadrics.
Note that $X$ satisfies uniform $h^0$-growth for curves of $n$-type provided that $\kappa(X,Q_1+Q_2)=1$.
This condition, however, is not necessary.
More generally, it remains unknown whether there exists an odd fake quadric $X$ with $b(X)=0$ such that $X$ satisfies uniform $h^0$-growth for curves of $n$-type but $\kappa(X,Q_1+Q_2)\le 0$.
\end{remark}
The paper is organized as follows.
In Section \ref{Pre},  we organize preliminary material into three subsections: preliminary results, fake quadrics, and the proof of Proposition \ref{MDS-rho2-prop}.
We prove Theorem \ref{Mainthm} in Section \ref{Main-Sect}.
\subsection*{Acknowledgments}
The author would like to thank Antonio Laface for helpful conversations, Yifan Chen for answering questions, and the referees for their valuable comments and suggestions.
 \section{Preliminaries}
\label{Pre}
{\bf Notation and Terminology.}
In this paper, $X$ is a smooth projective surface over $\bC$.
\begin{itemize}
\item By a curve on $X$, we mean a reduced and irreducible curve.
\item A negative curve on $X$ is a curve with negative self-intersection.
 \item A prime divisor $C$ on $X$ is either a nef curve or a negative curve (in the latter case, $h^0(\mathcal O_X(C))=1$).
\item We say $X$ has $b(X)>0$ if $X$ contains at least one negative curve.
\item We say $X$ has $b(X)=0$ if $X$ contains  no negative curves.
 \item For every $\bR$-divisor $C$ with $C^2\ne0$ on $X$, we define a value $l_C$ associated to $C$ as follows:
\begin{equation*}
                                                       l_C:=\frac{(K_X\cdot C)}{\max\bigg\{ 1, C^2\bigg\}}.
\end{equation*}
\item For every $\bR$-divisor $C$ with $C^2=0$ on $X$, we define a value $l_C$ associated to $C$ as follows:
\begin{equation*}
                                   l_C:=\frac{(K_X\cdot C)}{\max\bigg\{1,h^0(\mathcal O_X(C))\bigg\}}.
\end{equation*}
\item We say that the closed Mori cone \(\NE(X)\) is rational polyhedral if there exist finitely many curves $C_1,\cdots, C_s$ on $X$ such that $\NE(X)=\sum \bR_{\ge0}[C_i]$.
\end{itemize}
\subsection{Preliminary results}
\begin{proposition}
\label{RR-prop}
Let $X$ be a smooth projective surface.
Take a curve $C$ on $X$.
Then 
\begin{equation}
\label{RR-eq}
2h^1(\mathcal O_X(C))\le 2h^0(\mathcal O_X(C))+(K_X-C)\cdot C+2q(X)-2.
\end{equation}
\end{proposition}
\begin{proof}
By Serre duality, we have $h^2(\mathcal O_X(C))-\chi(\mathcal O_X)\le q(X)-1$.
This result and Riemann-Roch theorem imply that (\ref{RR-eq}) holds.
\end{proof}
\begin{proposition}
\label{Prop-Numerical}
\cite[Proposition 1.5]{HL26}
Let $X$ be a smooth projective surface.
If $X$ satisfies the BNC and there exists a positive constant $m(X)$ such that $l_C\le m(X)$ for every curve $C$ on $X$ and $D^2\le m(X)h^0(\mathcal O_X(D))$ for every curve $D$ with $D^2>0$ on $X$, then $X$ satisfies the BCP.
\end{proposition}
\begin{proposition}
\cite[Proposition 2.7]{Li23}
\label{C^2=0}
Let $C$ be a curve on a smooth projective surface $X$.
Suppose $C^2=0$.
Then there exists a constant $a_C$ such that $l_D\le a_C$ for every $n>0$ and  every curve $D\in |nC|$.
\end{proposition}
\begin{proposition}
\label{MDS-prop}
Let $X$ be a smooth projective surface with $q(X)=0$ and rational polyhedral Mori cone.
Then $X$ is a Mori dream surface if and only if every extremal ray of $\Nef(X)$ is generated by an semiample divisor.
\end{proposition}
\begin{proof}
It follows from \cite[Propositions 2.1 and 2.2]{LLL26}.
\end{proof}
\subsection{Fake quadrics}
In this subsection, we first recall the definitions and basic results of fake quadrics from \cite[Section 2]{CCCK25}.
\begin{definition}
Let $X$ be a smooth projective surface over $\bC$.
The surface $X$ is called \emph{$\bQ$-homology quadric} if $q(X)=p_g(X)=0, b_2(X)=2$.
It is called an \emph{even homology quadric} if 
\begin{enumerate}
\item[(1)] $q(X)=p_g(X)=0, b_2(X)=0$.
\item[(2)] the intersection form on $\NS(X)$ is even, that is, it is $\begin{pmatrix}0 & 1 \\ 1 & 0\end{pmatrix}$.
\end{enumerate}
We  call it an \emph{odd homology quadric} if instead the intersection form is odd, and hence diagonalizable with diagonal entries $(+1,-1)$.
\end{definition}
As showed in \cite[Section 2]{CCCK25}, for  every $\bQ$-homology quadric $X$ is either a Hirzebruch surface or a surface of general type, and 
$$
\chi(\mathcal O_X)=1,\quad  c_2(X)=4, \quad K_X^2=8.
$$
Note that Hirzebruch surfaces satisfy the BCP by \cite[Theorem 1.10(2)]{HL26}.
\begin{definition}
Let $X$ be a smooth projective surface.
The surface $X$ will be called an \emph{even fake homology quadric} if it is an even homology quadric and is of general type.
\end{definition}
In general, a minimal smooth complex projective surface of general type with $p_g=q=0$, $K^2=8$ has Picard number $2$ and $K_S$ is ample by \cite[Proposition 2.1.1]{Miyaoka84}. 
\begin{definition}
Let $X$ be a smooth projective surface. 
If $X$ is an odd homology quadric which is of general type, then either $X$ is not minimal and it is a one point blow up of a fake projective plane, or we call the surface $X$ an \emph{odd fake homology quadric}, which means:
\begin{enumerate}
    \item $S$ is minimal of general type with $K_X^2=8$, $p_g(X)=0$;
    \item the intersection form on $\NS(X)$ is odd.
\end{enumerate}
\end{definition}
Therefore smooth minimal surfaces of general type with $K^2=8$ and $p_g=0$ are divided into two classes: the even and the odd fake homology quadrics.
\begin{definition}
Let $X$ be a smooth projective surface. 
The surface $X$ is said to be \emph{isogenous to a higher product} \cite{Catanese00} if
$
X \cong (C_1 \times C_2)/G,
$
where $C_i$ is a smooth curve with $g(C_i)\ge 2$ and $G$ is a finite group acting faithfully and freely on $C_1\times C_2$.
If there are respective actions of $G$ on $C_1$ and $C_2$ such that $G$ acts by the diagonal action $g(x,y)=(gx,gy)$ on $C_1\times C_2$, then $X$ is called of \emph{unmixed type}.
If some element of $G$ exchanges the two factors, then $X$ is called of \emph{mixed type}.
\end{definition}
\begin{proposition}
Let $X$ be a surface isogenous to a higher product with $p_g=0$.
Then $X$ is a Mori dream surface.
\end{proposition}
\begin{proof}
It follows from \cite[Proposition 3.4]{KL19} and \cite[Theorem 3.5]{FL22}.
\end{proof}
\begin{proposition}
Let $X$ be a fake quadric, and let $f: X\to\bP^1$ be a fibration, whose general fibre we denote by $F$.
Then $X$ satisfies the BCP whenever $b(X)>0$.
\end{proposition}
\begin{proof}
By \cite[Proposition 6.2 (2)]{CCCK25}, $X$ has exactly one  negative curve $C$.
Then $\NE(X)=\bR_{\ge0}[F]+\bR_{\ge0}[C]$.
Note that $\kappa(X,F)=1$ and $q(X)=0$.
Thus, $X$ satisfies the BCP by \cite[Lemma 3.7]{Li23}.
\end{proof}
\subsection{Proof of Proposition \ref{MDS-rho2-prop}}
In this subsection, let $X$ be a smooth projective surface  with $q(X)=0$ such that $\NE(X)=\bR_{\ge0}[C_1]+\bR_{\ge0}[C_2]$, and $C_1$ and $C_2$ are curves on $X$.
\begin{proof}[Proof of Proposition \ref{MDS-rho2-prop}]
Note that $C_1^2\le0$ and $C_2^2\le0$ by \cite[Proposition 3.1]{Li21}.
Then $X$ fails into the following three cases: (1) $C_1^2=C_2^2=0$, (2) $C_1^2=0$ and $C_2^2<0$, and (3) $C_1^2<0$ and $C_2^2<0$.
By Proposition \ref{MDS-prop}, $X$ is a Mori dream surface if and only if each extremal ray of $\Nef(X)$  is generated by an semiample divisor.
Since $q(X)=0$, a nef isotropic divisor  $D$ has $\kappa(X,D)=1$  if and only if  it is semiample by \cite[Lemma 5.1.2.2]{ADHL15}.
By \cite[Proposition 1.4.28]{Lazarsfeld04}, $\NE(X)$ is dual to $\Nef(X)$.
Then we have the following result.
\begin{enumerate}
	\item If $C_1^2=C_2^2=0$, then $$\Nef(X)=\bR_{\ge0}[C_1]+\bR_{\ge0}[C_2].$$
	\item If $C_1^2=0$ and $C_2^2<0$, then $$\Nef(X)=\bR_{\ge0}[C_1]+\bR_{\ge0}[(-C_2^2)C_1+(C_1\cdot C_2)C_2].$$
	\item If $C_1^2<0$ and $C_2^2<0$, then 
	$$\Nef(X)=\bR_{\ge0}[(C_1\cdot C_2)C_1+(-C_1^2)C_2]+\bR_{\ge0}[(-C_2^2)C_1+(C_1\cdot C_2)C_2].$$
\end{enumerate}
This completes the proof of Proposition \ref{MDS-rho2-prop} by using Proposition \ref{MDS-prop}.
\end{proof}
\section{Proof of Theorem \ref{Mainthm}}
\label{Main-Sect}
\subsection{Even fake quadrics}
\begin{lemma}
\label{even-lem}
Every even fake quadric satisfies the BCP.
\end{lemma}
\begin{proof}
Since $\rho(X)=2$, $X$ has at most two negative curves.
To show the BCP for $X$, it suffices to show that there exists a constant $c_X>0$ such that $h^1(\mathcal O_X(D))\le c_Xh^0(\mathcal O_X(D))$ for every curve $D$ with $D^2\ge0$ on $X$.
By \cite[Lemma 4.1 and Corollary 4.3]{CCCK25}, there exist two nef divisors $C_1$ and $C_2$ such that $$C_1^2=C_2^2=0,\quad  C_1\cdot C_2=1,\quad  K_X\sim 2(C_1+C_2), \quad
\Nef(X)=\NE(X)=\bR_{\ge0}[C_1]+\bR_{\ge0}[C_2].
$$
Since $D, C_1$ and $C_2$ are nef divisors, $D\cdot C_1, D\cdot C_2\in\bZ_{\ge0}$.
Thus, $D=a_1C_1+a_2C_2$ with  $a_1,a_2\in\bZ_{\ge0}$.
Without loss of generality, we may assume that $a_1\ge a_2\ge0$.
If $a_2=0$, then $D=a_1C_1$, so $\kappa(X,C_1)\ge0$ since $D$ is a curve.
By Proposition \ref{C^2=0}, there exists a constant $m_1>0$ such that $l_D\le m_1h^0(\mathcal O_X(D))$.
Then by Proposition \ref{Prop-Numerical},  we may assume that $a_2\ge1$.
As a result, we have
\begin{equation}
\begin{split}
\label{eq-I}
(K_X-D)\cdot D&=(2-a_1)a_2+(2-a_2)a_1
\\&=2(a_1+a_2-a_1a_2).
\end{split}
\end{equation}
If $a_2\ge2$, then $(K_X-D)\cdot D\le 0$.
If $a_2=1$, then $(K_X-D)\cdot D=2$.
Thus, Proposition \ref{RR-prop} implies that  $h^1(\mathcal O_X(D)) \le h^0(\mathcal O_X(D))$, where we use $q(X)=0$.
\end{proof}
\subsection{Odd fake quadrics}
Let $X$ be an odd fake quadric.
By \cite[Lemma 6.1]{CCCK25}, there $Q_1,Q_2\in\Pic(X)$ such that $$Q_1^2=1, \quad Q_2^2=-1, \quad Q_1\cdot Q_2=0, \quad K_X=3Q_1-Q_2.$$
\begin{lemma}
\label{lem-odd-b>0}
Every odd fake quadric $X$ satisfies the BCP provided that $b(X)>0$.
\end{lemma}
\begin{proof}
Since $b(X)>0$, $X$ has exactly one negative curve $C=aQ_1+(a+1)Q_2$ with $a\in\bZ_{\ge0}$ by \cite[Proposition 6.2]{CCCK25}.
To show the BCP for $X$, it suffices to show that there exists a constant $c_X>0$ such that either $D^2\le c_Xh^0(\mathcal O_X(D))$ or $h^1(\mathcal O_X(D))\le c_Xh^0(\mathcal O_X(D))$ for every curve $D$ with $D^2\ge0$ on $X$.
Note that another extremal ray of $\NE(X)$ is generated by  $Q_1-Q_2$ or $Q_1+Q_2$.
Note that $(Q_1-Q_2)\cdot C=2a+1$ and $(Q_1+Q_2)\cdot C=-1$.
Then $(Q_1-Q_2)$ generates another extremal ray of $\NE(X)$ since $\rho(X)=2$.
Indeed, choose a curve $D=a_1(Q_1+Q_2)+a_2C$ with $a_1, a_2>0$.
Hence, $D^2=-2a_1a_2-a_2^2<0$, which is absurd.
Thus, let $C_1=Q_1-Q_2, C_2=C$, and 
 $$\NE(X)=\bR_{\ge0}[C_1]+\bR_{\ge0}[C_2].$$
Note that $$C_1^2=0, C_2^2=-1, C_1\cdot C_2=2a+1, K_X=(2+\frac{1}{2a+1})C_1+\frac{2}{2a+1}C_2,D=a_1C_1+a_2C_2$$ with $a_1,a_2\ge0$.
Set $$b=(2+\frac{1}{2a+1}), \quad b'=\frac{2}{2a+1}.$$
If $D^2=0$, then $D=a_1C_1$, and $\kappa(X,C_1)=1$ since $D$ is a curve.
By Proposition \ref{C^2=0}, there exists a constant $m_1>0$, independent of  $D$, such that $l_D\le m_1h^0(\mathcal O_X(D))$.
Then by Proposition \ref{Prop-Numerical}, we may assume that $a_1>0$ and $a_2>0$.
Since $D, C_1$ and $C_2$ are nef divisors, then $D\cdot C_1=(2a+1)a_2\in\bZ_{>0}$ and $D\cdot C_2=(2a+1)a_1-a_2\in\bZ_{>0}$.
As a result, we have
\begin{equation}
\label{a_2<a_1}
a_1\in\bZ_{>0}, \quad a_2\in\frac{1}{2a+1}\bZ_{>0}, \quad  a_2<a_1(2a+1).	
\end{equation}
If $a_1\ge b$ and $a_2\ge b'$, then 
$$
(K_X-D)\cdot D=(b-a_1)(D\cdot C_1)+(b'-a_2)(D\cdot C_2)\le0.
$$
Thus, Proposition \ref{RR-prop} implies that $h^1(\mathcal O_X(D))  \le h^0(\mathcal O_X(D))$.
If $a_1\le b$, then $a_2\le b(2a+1)$ by (\ref{a_2<a_1}).
As a result, 
\begin{equation*}
\begin{split}
	D^2 &\le 2a_1a_2(C_1\cdot C_2)
	\\&\le 2b^2(2a+1)^2.
\end{split}	
\end{equation*}
Now we may assume that $a_1>b$ and $a_2<b'$.
Note that $2a_2\ge b'$.
Then 
\begin{equation}
\begin{split}
\label{eq-II}
(K_X-D)\cdot D&=((b-a_1)C_1+(b'-a_2)C_2)\cdot (a_1C_1+a_2C_2)
\\&=(ba_2+a_1(b'-2a_2))(C_1\cdot C_2)+a_2(a_2-b')(-C_2^2)
\\&\le bb'(2a+1).
\end{split}
\end{equation}
Therefore, Proposition \ref{RR-prop} implies that
\begin{equation*}
\begin{split}
2h^1(\mathcal O_X(D)) & \le (2+bb'(2a+1))h^0(\mathcal O_X(D)).
\end{split}
\end{equation*}
In all, $X$ satisfies the BCP.
\end{proof}
\begin{lemma}
\label{lem-reduced}
Let $X$ be an odd fake quadric with $b(X)=0$.
Then $X$ satisfies the BCP provided that $I_X$ is finite.
\end{lemma}
\begin{proof}
Since $b(X)=0$, to show the BCP for $X$, it suffices to show that there exists a constant $c_X>0$ such that $h^1(\mathcal O_X(D))\le c_Xh^0(\mathcal O_X(D))$ for every curve $D$ with $D^2\ge0$ on $X$.
Since $b(X)=0$, then $$\NE(X)=\bR_{\ge0}[Q_1-Q_2]+\bR_{\ge0}[Q_1+Q_2]$$ by \cite[Proof of Theorem 4.5]{CCCK25}.
Set $C_1:=Q_1-Q_2$ and $C_2:=Q_1+Q_2$.
Then $$C_1^2=C_2^2=0, \quad C_1\cdot C_2=2, \quad K_X=2C_1+C_2, \quad D=a_1C_1+a_2C_2$$ with $a_1,a_2\ge0$.
If $D^2=0$, then $D=a_1C_1$ or $a_2C_2$.
Without loss of generality, we may assume that $D=aC_1$, and $\kappa(X,C_1)=1$ since $D$ is a curve.
By Proposition \ref{C^2=0}, there exists a constant $m_1>0$, independent of  $D$,  such that $l_D\le m_1h^0(\mathcal O_X(D))$.
Then we may assume that $a_1>0$ and $a_2>0$ by Proposition \ref{Prop-Numerical}.
Since $D, C_1$ and $C_2$ are nef divisors, then $D\cdot C_1=2a_2\in\bZ_{>0}$ and $D\cdot C_2=2a_1\in\bZ_{>0}$.
As a result, we have
\begin{equation}
\begin{split}
\label{eq-II}
(K_X-D)\cdot D&=(2-a_1)2a_2+(1-a_2)2a_1
\\&=2(a_1+2a_2-2a_1a_2).
\end{split}
\end{equation}
Note that $a_1\in\frac{1}{2}\bZ_{>0}$ and $a_2\in\frac12\bZ_{>0}$.

If $a_1\ge2$ and $a_2\ge1$, then $(K_X-D)\cdot D\le0$.

If $a_1\ge2$ and $\frac 12\le a_2<1$, then $(K_X-D)\cdot D=2((1-2a_2)a_1+a_2)\le1$.

If $1\le a_1<2$, then  $(K_X-D)\cdot D=2(2(1-a_1)a_2+a_1)\le 2$.

If $a_1=a_2=\frac 12$, then $(K_X-D)\cdot D=2$.
Thus, Proposition \ref{RR-prop} implies that $$h^1(\mathcal O_X(D)\le h^0(\mathcal O_X(D))$$ if either $a_1\ge1$ or $a_1=a_2=\frac 12$.

If  $a_1=\frac 12$ and $a_2\ge1$, then $(K_X-D)\cdot D=2a_2+1$.
As a result, $D=(\frac{1}{2}+a_2)Q_1+(a_2-\frac{1}{2})Q_2$.
Since $\NS(X)=\bZ[Q_1]+\bZ[Q_2]$, $a_2-\frac12=n\in\bZ_{>0}$.
Then $D=(n+1)Q_1+nQ_2$ is a curve of $n$-type.
Since $I_X$ is finite, we may take
$$m(X):=\max_{n\in I_X}\bigg\{\frac{1+ h^1(\mathcal O_X((n+1)Q_1+nQ_2)))}{h^0(\mathcal O_X((n+1)Q_1+nQ_2)))}\bigg\}.$$
Then $h^1(\mathcal O_X(D))\le m(X)h^0(\mathcal O_X(D))$ for every curve $D$ of $n$-type on $X$.
\end{proof}
\begin{lemma}
\label{lem-odd-b=0}
If $X$ is of odd type with $b(X)=0$ and $I_X$ is infinite, then $X$ satisfies the BCP if and only if $X$ satisfies uniform $h^0$-growth for curves of $n$-type.
\end{lemma}
\begin{proof}
$\Rightarrow$: Since $X$ satisfies the BCP, there exists a constant $c_X>0$ such that  
\begin{equation}
\label{BCP-eq}
h^1(\mathcal O_X(D))\le c_Xh^0(\mathcal O_X(D)),\quad c_X>1
\end{equation}
 for every curve $D=(n+1)Q_1+nQ_2$ on $X$.
Note that 
\begin{equation}
\label{KD-eqI}
(K_X-D)\cdot D=2n+2.
\end{equation}
 Riemann-Roch theorem implies that
 \begin{equation}
 \label{RR-eqII}
 2h^1(\mathcal O_X(D))=2h^0(\mathcal O_X(D))+(K_X-D)\cdot D+2h^2(\mathcal O_X(D))-2\chi(\mathcal O_X).
 \end{equation}
 Note that $2h^2(\mathcal O_X(D))-2\chi(\mathcal O_X)\le q(X)-1$.
Since $I_X$ is infinite,   (\ref{BCP-eq}), (\ref{KD-eqI}) and (\ref{RR-eqII}) imply that 
\begin{equation*}
\liminf_{n\to +\infty}\frac{h^0(\mathcal O_X((n+1)Q_1+nQ_2)))}{n}\ge\frac{1}{2(c_X-1)}\liminf_{n\to+\infty}\frac{(K_X-D)\cdot D}{n}>0. 
\end{equation*}

$\Leftarrow$:  By the proof of Lemma \ref{lem-reduced}, to show that BCP for $X$, it suffices to show that there exists a postive constant $m(X)>0$ such that $h^0(\mathcal O_X(D))\le c_Xh^0(\mathcal O_X(D))$ for every curve $D$ of $n$-type on $X$.
Suppose $X$ satisfies uniform $h^0$-growth for  curves of $n$-type.
Then  
\begin{equation}
\label{liminf-eq}
\liminf_{\substack{n \in I_X,~ n\to+\infty}}\frac{h^0\bigl(\mathcal O_X\big((n+1)Q_1+nQ_2\big)\bigr)}{n}>0.	
\end{equation}
Note that (\ref{liminf-eq}) implies that there exists a postive constant $N$ and constant $k>0$ such that for $n>N$, we have 
\begin{equation}
\label{kh-eqI}
	kh^0(\mathcal O_X((n+1)Q_1+nQ_2)))\ge n+1+\chi(\mathcal O_X)-h^2(\mathcal O_X(D)).
\end{equation}
Then Proposition \ref{RR-prop}, (\ref{KD-eqI}) and (\ref{kh-eqI})  imply that
$$
2h^1(\mathcal O_X(D))\le (k+2)h^0(\mathcal O_X(D)).
$$ 
If $n\le N$, then  Proposition \ref{RR-prop} and (\ref{KD-eqI}) imply that 
$$
2h^1(\mathcal O_X(D))\le (2+2N)h^0(\mathcal O_X(D)).
$$
In all, $X$ satisfies the BCP.
\end{proof}
\subsection{Proof of Theorem \ref{Mainthm}}
\begin{proof}[Proof of Theorem \ref{Mainthm}]
It follows from Lemmas \ref{even-lem}, \ref{lem-odd-b>0}, \ref{lem-reduced} and  \ref{lem-odd-b=0}.
\end{proof}

\end{document}